\documentclass{article}
\usepackage{amssymb,amsmath,amsfonts, amsthm, amscd}
\usepackage{arydshln}
\usepackage{float}
\usepackage{graphicx}
\usepackage{caption}
\usepackage{rotating}
\usepackage{pdflscape}
\usepackage{tikz}
\usepackage[utf8]{inputenc}
\usepackage{booktabs}
\usepackage{lipsum}
\usepackage{multirow}
\usepackage{enumitem}
\usepackage{subcaption}
\usepackage{multirow}
\usepackage{epsfig,amsthm,amsmath,amsfonts,amssymb}
\usepackage{fancyhdr,multicol,color,xcolor}
\usepackage[english]{babel}
\usepackage[utf8]{inputenc}
\usepackage{algorithm}
\usepackage[noend]{algpseudocode}

\numberwithin{equation}{section} \numberwithin{equation}{section}

\label{sec.sub.gdn}

\newtheorem{definition}{Definition}
\newtheorem{corollary}{Corollary}
\newtheorem{ex}{Example}
\newtheorem{theorem}{Theorem}

\newtheorem{remark}{Remark}

\numberwithin{equation}{section}

\begin{document}

	\title{\bf  %(A Framework to)
		Solving linear systems over idempotent semifields through $LU$-factorization}
		\author{{Sedighe Jamshidvand\textsuperscript{$a$}, Shaban Ghalandarzadeh\textsuperscript{$a$}, Amirhossein Amiraslani\textsuperscript{$b,a$}, Fateme Olia\textsuperscript{$a$}}\\
		{\em \small   $^1$ Faculty of Mathematics, K. N. Toosi University of
			Technology, Tehran, Iran} \\
		{\em \small   $^2$ STEM Department, The University of Hawaii-Maui College, Kahului, Hawaii, USA }}
	\maketitle{}	
	
\begin{abstract}
	In this paper, we introduce and analyze a new $LU$-factorization technique for square matrices over idempotent semifields. In particular, more emphasis is put on ``max-plus'' algebra here, but the work is extended to other idempotent semifields as well. We first determine the conditions under which a square matrix has $LU$ factors. Next, using this technique, we propose a method for solving square linear systems of equations whose system matrices are $LU$-factorizable. We also give conditions for an $LU$-factorizable system to have solutions. This work is an extension of similar techniques over fields. Maple procedures for this $LU$-factorization are also included.
\end{abstract}
{\small {\it key words}: Semiring; Idempotent semifield; $LU$-factorization; Linear system of equations }\\[0.3cm]

\section{Introduction}
Linear algebra provides significantly powerful problem-solving tools. Many problems in the mathematical sciences consist of systems of linear equations in numerous unknowns. Nonlinear and complicated problems often can be treated by linearization and turned into linear problems that are the ``best possible linear approximation''. There are some well-known symbolic and numeric methods and results for the analysis and solution of linear systems of equations over fields and rings. In this work, we propose and analyze a method for solving linear systems of equations over idempotent semifields based on our new $LU$-factorizatin technique. Note that idempotent semifields are a special class of semirings. This work is at the intersection of numerical linear algebra, and pure mathematics.

Semirings are a generalization of rings and lattices. These algebraic structures are similar to rings where division and subtraction are not needed or can not be defined. Semirings have applications in various areas of mathematics and engineering such as formal language, computer science, optimization theory, control theory, etc. (see e.g.~\cite{hebish}, \cite{control}, \cite{key}, \cite{butkovic}).

The first definition of a semiring was given by Vandiver \cite{vandiver} in 1934. A semiring $(S, \oplus, \otimes,0,1)$ is an algebraic structure in which $(S, \oplus)$ is a commutative monoid with an identity element $0$ and $(S,\otimes)$ is a monoid with an identity element 1, connected by ring-like distributivity. The additive identity $0$ is multiplicatively absorbing, and $0 \neq 1$.

Some articles have touched on techniques for $ LU $ decomposition over tropical semirings (see~\cite{tan2}, \cite{inrsln}, \cite{incomp}). In this work, we present a new $LU-$factorization technique for a matrix using its entries. This approach enables us to solve systems of linear equations over idempotent semifields applying our $ LU $ factors. Note that similar to the traditional linear algebra, we set the main diagonal entries of $ L $ to $ 1 $ for simplicity. We also show that the solution through this method is maximal. A maximal solution is determined with respect to the defined total order on the idempotent semifield. We consider the linear system of equations, $AX=b$, where $A=(a_{ij})$ is a square matrix and $b=(b_{i})$ is a column vector over an idempotent semifield.  A semifield is a commutative semiring such that every nonzero element of it is invertible. We say a matrix $A$ has an $LU$-factorization if there exist a lower triangular matrix $L$ and an upper triangular matrix $U$ such that $A=LU$.

In \cite{tan2}, Tan shows that a square matrix $A$ over a commutative semiring has an $LU$-factorization if and only if every leading principle submatrix of $A$ is invertible. Moreover, Sararnrakskul proves in \cite{sara}  that a square matrix $A$ over a semifield is invertible if and only if every row and every column of it contains exactly one nonzero element. The same proposition has been proved for certain semifields , namely Boolean algebra and tropical semirings, in~\cite{rutherford} and~\cite{bacceli}. This means that only diagonal matrices are $LU$-factorizable in  semifields and a fortiori in idempotent semifields.

In traditional linear algebra, $LU$ decomposition is the matrix form of Gaussian elimination. Since our goal here is to take advantage of the decomposition technique for solving a system of linear equations, we define a new $LU-$factorization that can be used for an arbitrary non-diagonal matrix, $A$, in ``max-plus algebra''. We give criteria for the existence of a lower triangular matrix, $L_A$, and an upper triangular matrix, $U_A$, such that $A=L_{A} U_{A}$. We call $L_A$ and $U_A$ the  $LU$ factors of $A$. See Section~3 for more details.

In Section~4, assuming that the matrix $A$ can be written in an $LU$ form, we use the factorization to solve the system of linear equations $AX=b$ for $X$. In fact, $ AX=b $ may be rewritten as $ LUX= b $. To find $ X $, we must first solve the system $ LZ=b $ for $ Z $, where $UX=Z$. Once $ Z $ is found, we solve the system $UX=Z $ for $ X $. We present some theorems on when $LZ=b$ and $UX=Z$ have solutions.

Section $5$ concerns the extension of the proposed $LU-$decomposition and system solution from max-plus algebra to idempotent semifields.

In the appendix of this paper, we give some Maple procedures as follows. Tables $1$ and $2$ are subroutines for finding the determinant of a square matrix, and its permutation matrix, respectively. Table~3 gives a code for calculating matrix multiplications. Tables $5$ and $6$ consist of programs for solving the systems $LX=b$ and $UX=b$, respectively. These procedures are all written for max-plus algebra.\\
We now proceed with some basic definitions and theorems in Section~2.
\section{Definitions and preliminaries}
In this section, we go over some definitions and preliminary notions. For convenience, we use $ \mathbb{N}$, and $\underline{n}$ to denote the set of all positive integers, and the set $\{ 1,2,\cdots,n\}$ for $n \in \mathbb{N}$, respectively.
\begin{definition}A semiring $ (S,\oplus,\otimes, 0,1)$ is an algebraic system consisting of a nonempty set $ S $ with two binary operations,  addition and multiplication, such that the following conditions hold:
	\begin{enumerate}
		\item{$ (S, \oplus) $ is a commutative monoid with identity element $ 0 $};
		\item{ $ (S, \otimes) $ is a monoid with identity element $ 1 $};
		\item{ Multiplication distributes over addition from either side, that is $ a \otimes(b\oplus c)= (a \otimes b) \oplus (a \otimes c )$ and $ (b \oplus c)\otimes a=(b \otimes a) \oplus (c\otimes a) $ for all $ a, b \in S $};
		\item{ The neutral element of $ S $ is an absorbing element, that is $ a \otimes 0 =0= 0 \otimes a  $ for all $ a \in S $};
		\(\)
		\item{ $ 1 \neq 0 $}.
	\end{enumerate}
	A semiring is called commutative if $ a \otimes b = b \otimes a $ for all $ a, b \in S $.
\end{definition}
\begin{definition}	A semiring $S$ is called zerosumfree if $a \oplus b=0$ implies that $ a=0=b $, for any $ a,b \in S $.
\end{definition}
\begin{definition}
	Let $S$ be a semiring. An element $a \in S$ is additively idempotent if and only if $ a \oplus a=a $. A semiring $S$ is called additively idempotent if every element of $S$ is additively idempotent.\\
\end{definition}
\begin{definition}
	A commutative semiring $(S,\oplus,\otimes,0,1)$ is called a semifield if every nonzero element of $S$ is multiplicatively invetrible.
\end{definition}
\begin{definition}
	The semifield $(S,\oplus, \otimes, 0,1)$ is idempotent if it is an additively idempotent, totally ordered, and radicable semifield. Note that the radicability implies the power $a^{q}$ be defined for any $a \in S \setminus \{0\}$ and $q \in \mathbb{Q}$ (rational numbers). In particular, for any non-negative integer $p$ we have
	$$a^{0}=1 ~~~~~~~ a^{p}=a^{p-1}a~~~~~ a^{-p}=(a^{-1})^{p} $$
\end{definition}
The totally ordered operator ``$\leq_{S} $" on an idempotent semifield  defines the following partial order:
$$ a \leq_{S} b \Longleftrightarrow a \oplus b = b,$$
that is induced by additive idempotency. Notice that the last partial order equips addition to an extremal property in the form of the following inequalities
$$ a \leq_{S} a \oplus b ,~~~~~~ b\leq_{S} a \oplus b ,$$
which makes idempotent semifields zerosumfree, since we have $a \geq_{S} 0$ for any $ a  \in S$.\\
The total order defined on $ S $ is compatible with the algebraic operations, that is:
$$ (a \leq_{S} b \wedge c \leq_{S} d) \Longrightarrow a\oplus c \leq_{S} b \oplus d, $$
$$ (a \leq_{S} b \wedge c \leq_{S} d) \Longrightarrow a\otimes c \leq_{S} b \otimes d,$$
for any $ a,b,c,d \in S $.

Throughout this article, we consider the notation $``\geq_{S}" $ as the converse of the order $ ``\leq_{S}" $, which is a totally ordered operator satisfying $ a\geq_{S} b $ if and only if $ b\leq_{S} a $, for any $ a, b \in S $. Furthermore, we use  $ ``a <_{S} b" $ whenever $ a\leq_{S} b $ and $ a \neq b$. The notation $ ``>_{S} " $ is defined similarly.
\begin{ex}
	Tropical semirings are important cases of idempotent semifields that we consider in this work. Examples of tropical semirings are as follows.
	\begin{align*}
	\mathbb{R}_{\max, +}=(\mathbb{R} \cup \{-\infty\}, \max,+,-\infty,0),~\\
	\mathbb{R}_{\min, +}=(\mathbb{R} \cup \{+\infty\}, \min,+,+\infty,0),~~\\
	\mathbb{R}_{\max,\times}=(\mathbb{R_{+}} \cup \{0\}, \max,\times,0,1),~~~~~~\\
	\mathbb{R}_{\min,\times}=(\mathbb{R_{+}} \cup \{+\infty\}, \min,\times,+\infty,1),\\
	\end{align*}
	where $\mathbb{R}$ is the set of real numbers, and $\mathbb{R_{+}}=\{ x \in \mathbb{R} \vert x > 0 \}$. This work particularly concerns $\mathbb{R}_{\max, +}$	that is called ``$\max-\rm plus$ algebra" whose additive and multiplicative identities are $-\infty$ and $0$, respectively. Note further that the multiplication  $ a \otimes b^{-1} $ in ``$\max-\rm plus$ algebra" means that $ a + (-b) =a-b $, where $ ``+" ,~``-" $ and $ -b $ denote the usual real numbers addition , subtraction and the typical additively inverse of the element $ b $, respectively. Moreover, the defined total order on ``$\max-\rm plus$ algebra" is the standard less than or equal relation $ `` \leq " $ over $\mathbb{R}$.
\end{ex}
The semifields $\mathbb{R}_{\max, +}$, $\mathbb{R}_{\min, +}$, $\mathbb{R}_{\max,\times}$ and $\mathbb{R}_{\min,\times}$ are isomorphic to one other. Figure~\ref{fig1} shows the isomorphism maps for these semifields.
\begin{figure}[hbt!]\centering
	\begin{tikzpicture}
	\draw[->](0,0.2)--(3,0.2);
	\draw[->](3,0)--(0,0);
	\draw (1.4,.4) node{$y = -x$};
	\draw (1.4,-.4) node{$y = -x$};
	\draw[->](4.2,-0.4)--(4.2,-3.4);
	\draw[->](4,-3.4)--(4,-.4);
	\draw (4.8,-1.9) node{$y = e^x$};
	\draw (3.3,-1.9) node{$y = ln x$};
	\draw (4,0.1) node{$\mathbb{R}_{\min, +}$};
	\draw[->](-.4,-.4)--(-.4,-3.4);
	\draw[->](-.6,-3.4)--(-.6,-.4);
	\draw (-.8,0.1) node {$\mathbb{R}_{\max, +}$};
	\draw (-1.3,-1.9) node{$ y = ln x$};
	\draw (.4,-1.9) node{$y = e^x$};
	\draw (-.8,-3.8) node {$\mathbb{R}_{\max,\times }$};
	\draw (4,-3.8) node {$\mathbb{R}_{\min, \times}$};
	\draw (1.4,-3.5) node{$y = 1/x$};
	\draw (1.4,-4.3) node{$y = 1/x$};
	\draw[->](0,-3.9)--(3,-3.9);
	\draw[->](3,-3.7)--(0,-3.7);
	\end{tikzpicture}
	\caption{Isomorphism maps for tropical semirings} \label{fig1}
\end{figure}

Let $S$ be a commutative semiring. We denote the set of all $m \times n$ matrices over $S$ by $M_{m \times n}(S)$.
For $A \in M_{m \times n} (S)$, we denote by $a_{ij}$ and $A
^{T}$ the $(i,j)$-entry of $A$ and the transpose of $A$, respectively.\\
For any $A, B \in M_{m \times n}(S)$, $C \in M_{n \times l}(S)$ and $\lambda \in S$, we define:
$$ A+B = (a_{ij} \oplus b_{ij})_{m \times n},$$
$$ AC=(\bigoplus_{k=1}^{n} a_{ik} \otimes b_{kj})_{m \times l},$$
and $$\lambda A=(\lambda \otimes a_{ij})_{m\times n}.$$

It is easy to verify that $M_{n}(S):=M_{n \times n}(S)$ forms a semiring with respect to the matrix addition and the matrix multiplication. Take matrices $A ,B \in M_{n}(S)$. We say that $A \leq B$ if $a_{ij} \leq b_{ij}$ for every $i, j \in \underline{n}$.\\
Let $A \in M_{n}(S)$, $\mathcal{S}_{n}$ be the symmetric group of degree $n \geq 2$ , and $\mathcal{A}_{n}$ be the alternating group on $n$ such that
\begin{center}
	$\mathcal{A}_{n}= \{ \sigma \vert \sigma \in \mathcal{S}_{n} ~\text{and}~\sigma \text{~is~an~even~permutation}\}$.
\end{center}
The positive determinant, $\vert A \vert ^{+}$, and negative determinant, $\vert A \vert ^{-}$, of $A$ are
$$\vert A \vert ^{+}=\bigoplus_{\sigma \in \mathcal{A}_{n}} \bigotimes_{i=1}^{n} a_{i\sigma (i)},$$
and
$$\vert A \vert ^{-}= \bigoplus_{\sigma \in \mathcal{S}_{n} \backslash \mathcal{A}_{n}} \bigotimes_{i=1}^{n} a_{i\sigma (i)}. $$
Clearly, if $S$ is a commutative ring, then $\vert A \vert=\vert A \vert ^{+} - \vert A \vert ^{-}$.
\begin{definition}
	Let $S$ be a semiring. A bijection $\varepsilon$ on $S$ is called an $\varepsilon$-function of $S$ if $\varepsilon(\varepsilon(a))=a$, $\varepsilon(a \oplus b)= \varepsilon(a) \oplus \varepsilon(b)$, and $\varepsilon(a \otimes b)=\varepsilon(a) \otimes b=a \otimes \varepsilon(b)$ for all $a, b \in S$. Consequently, $\varepsilon(a) \otimes \varepsilon(b)=a \otimes b$ and $\varepsilon(0)=0$.\\
	The identity mapping: $a \mapsto a$ is an $\varepsilon$-function of $S$ that is called the identity $\varepsilon$-function.
\end{definition}
\begin{remark}
	Any semiring $S$ has at least one $\varepsilon$-function since the identical mapping of $S$ is an $\varepsilon$-function of $S$. If $S$ is a ring, then the mapping : $a \mapsto -a$ $( a \in S)$ is an $\varepsilon$-function of $S$.
\end{remark}
\begin{definition}
	Let $S$ be a commutative semiring with an $\varepsilon$-function $\varepsilon$, and $A \in M_{n}(S)$. The $\varepsilon$-determinant of $A$, denoted by $det_{\varepsilon}(A)$, is defined by
	\begin{center}
		$$det_{\varepsilon}(A)= \bigoplus_{\sigma \in \mathcal{S}_{n}}\varepsilon ^{\tau(\sigma)}(a_{1\sigma(1)} \otimes a_{2\sigma(2)} \otimes \cdots \otimes a_{n\sigma(n)})$$
	\end{center}
	where $\tau(\sigma)$ is the number of the inversions of the permutation $\sigma$, and $\varepsilon^{(k)}$ is defined by $\varepsilon^{(0)}(a)=a$ and $\displaystyle{\varepsilon^{(k)}(a)=\varepsilon^{k-1}(\varepsilon(a))}$ for all positive integers $k$.
	Since $\varepsilon^{(2)}(a)=a$, $det_{\epsilon}(A)$ can be rewritten in the form of
	\begin{center}
		$det_{\varepsilon}(A)=\vert A \vert ^{+} \oplus \varepsilon(\vert A \vert^{-})$.
	\end{center}
	In particular, for $S=\mathbb{R}_{\max, +}$ with the identity $\varepsilon$-function $\varepsilon$, we  have $det_{\varepsilon}(A)=\max(\vert A \vert ^{+}, \vert A \vert ^{-})$.
\end{definition}
Let $(S,\oplus, \otimes, 0,1)$ be an idempotent semifield, $A \in M_{n}(S)$, and $b \in S^{n}$ be a column vector. Then the $i-$th equation of the linear system $AX=b$ is
$$ \bigoplus_{j=1}^{n}( a_{ij} \otimes x_{j})= (a_{i1}\otimes x_{1})\oplus(a_{i2}\otimes x_{2}) \oplus \cdots \oplus (a_{in} \otimes x_{n}) =b_{i}.$$
Especially, the $i-$th equation of the linear system over $\mathbb{R}_{\max, +}$ is
$$\max (a_{i1}+x_{1}, a_{i2}+x_{2}, \cdots, a_{in}+x_{n})=b_{i}.$$
\begin{definition}
	Let $b\in S^{n}$. Then $ b $ is called a regular vector if it has no zero element.
\end{definition}
Without loss of generality, we can assume that $b$ is regular in the system $AX=b$. Otherwise, let $b_{i}=0$ for some $i \in \underline{n}$. Then in the $i-$th equation of the system, we have $a_{ij}\otimes x_{j}=0$ for any $j \in \underline{n}$, since $S$ is zerosumfree. As such, $x_{j}=0$ if $a_{ij}\neq 0$.  Consequently, the $i-$th equation can be removed from the system together with every column $A_{j}$ where $a_{ij} \neq 0$, and the corresponding $x_{j}$ can be set to $0$.
\begin{definition}
	A solution $X^{*}$ of the system $AX=b$ is called maximal if $X \leq X^{*}$ for any solution $X$.
\end{definition}
\begin{theorem}\label{invert}(See~\cite{sara})
	Let $S$ be a semifield and $A \in M_{n}(S)$. Then $A$ is invertible if and only if every row and every column of $A$ contains exactly one nonzero element.
\end{theorem}
Let $\sigma \in \mathcal{S}_{n}$ be a permutation. Then $P_{\sigma}$ is the permutation matrix corresponding to $\sigma$. It is a square matrix obtained from the same size identity matrix through a permutation of the rows. It is clear that every row and every column of a permutation matrix contains exactly one nonzero element, therefore it is invertible according to Theorem~\ref{invert}. As a result, the system of equations $AX=b$ and $P_{\sigma}AX=P_{\sigma}b$ have the same solutions.\\
Throughout this paper, without loss of generality, we can assume that
$$det_{\varepsilon}(A)=a_{11}+\cdots+a_{nn};$$
otherwise, there exists a permutation matrix $P_{\sigma}$ such that
$$det_{\varepsilon}(P_{\sigma}A)=(P_{\sigma}A)_{11}+\cdots+(P_{\sigma}A)_{nn}.$$
Let $S$ be a commutative semiring, $A \in M_{n}(S)$ and $V , W \subseteq \underline {n}$. We denote by $A[V\vert W]$ the matrix with row indices indexed by indices in $V$, in an increasing order, and column indices indexed similarly by indices in $W$. The matrix $A[V \vert W]$ is called a submatrix of $A$. Particularly, $A[\underline{i}\vert \underline{i}]$ is called a leading principle submatrix of $A$ for any $i \in \underline{n}$.
\begin{theorem}\label{tanthm}(See~\cite{tan2}) Let $(S, \oplus, \otimes ,0,1)$ be a commuatative semiring. Then for any $A \in M_{n}(S)$, the following statements are equivalent.
	\begin{enumerate}
		\item All the leading principle submatrices of $A$ are invertible.
		\item A has an $LU$-factorization where $L$ is an invertible lower triangular matrix with $1$'s on its main diagonal, and $U$ is an invertible upper triangular matrix with $u_{ii}\in U(S)$ where $U(S)$ denotes the set of all multiplicatively invertible elements of $S$.
	\end{enumerate}
\end{theorem}
Let $S$ be an idempotent semifield and $A \in M_{n}(S)$. Suppose that $A$  has an $LU$-factorization as mentioned in Theorem~\ref{tanthm}. Then all the leading principle submatrices of $A$ are invertible. Moreover, per Theorem~\ref{invert}, we can obtain $a_{ii} \neq 0$ for any $i \in \underline{n}$. This means that $A$ has an $LU$-factorization if and only if it is a diagonal matrix with nonzero diagonal elements.\\
In what follows, we actually find new $L$ and $U$ factors for certain square matrices which are not necessarily diagonal, and use them to solve linear systems of equations.\\

Note that our main focus in the following two sections is on the ``$\max-\rm plus$ algebra".
\section{LU-factorization}
Take an arbitrary matrix $A=(a_{ij}) \in M_{n}(S)$. Let two matrices $L_{A}$ and $U_{A}$ be as follows.
\begin{equation}\label{L}
L_{A}=(l_{ij}) ; ~~
l_{ij}= \left
\{
\begin{array}{cc}
a_{ij}-a_{jj} & i\geq j;\\
-\infty &  i <j,
\end{array}
\right.	
\end{equation}
and
\begin{equation}\label{U}
U_{A}=(u_{ij}) ; ~~ u_{ij}= \left
\{
\begin{array}{cc}
a_{ij} & i\leq j;\\
-\infty & i > j.
\end{array}
\right.
\end{equation}
We say a matrix $ A \in M_{n}(S)$ has an $LU$-factorization if $A=L_{A}U_{A}$.
For simplicity, we show $L_{A}$ just by $L$ and $U_{A}$ just by $U$ from this point on. Clearly, $L$ and $U$ are lower and upper triangular matrices, respectively.\\

In the following theorem, we present the necessary and sufficient conditions for the existence of the presented $LU$ factors of an arbitrary square matrix $A$.
\begin{theorem}\label{LUable}
	Let $A \in M_{n}(S)$, $n$ be a positive integer such that $det_{\varepsilon}(A)=a_{11}+....+a_{nn}$ and the matrices $L$ and $U$ be defined by~\eqref{L} and~\eqref{U}, respectively. Then $A=LU$ if and only if for any $ 1 < i, j \leq n $ and $i\neq j$,
	\begin{equation}\label{LUcriteria}
	a_{ij}= \max_{k=1}^{r} ( det_{\varepsilon} (A[\{k,i\} \mid \{k,j\}])-a_{kk}),
	\end{equation}
	where $r=min\{i,j\}-1$.
\end{theorem}
\begin{proof}
	Suppose that the conditions~\eqref{LUcriteria} hold for any $ 1 < i, j \leq n $ and $i\neq j$. Then it suffices to show that for all $i,j \in \underline{n}$, $a_{ij}=(LU)_{ij}$. To this end, we consider the following cases:
	\begin{enumerate}
		\item Let $i=1$. Then for all $j \in \underline{n}$,
		\begin{align*}
		(LU)_{1j} &\ = \bigoplus_{k=1}^{n}(l_{1k} \otimes u_{kj})\\
		&\ =\max_{k=1}^{n}(l_{1k}+u_{kj})\\
		&\ =l_{11}+u_{1j}\\
		&\ =a_{1j},
		\end{align*}
		given that for $k > 1$, $l_{1k}=-\infty$.
		\item Similarly, for $j=1$, we have $(LU)_{i1}=a_{i1}$ for all $i \in \underline{n}$ since $u_{k1}=-\infty$ for all $k> 1$.
		\item Now assume that $i=j$ for all $i,j \in \underline{n}$. Then
		\begin{align}
		(LU)_{ii} &\ \nonumber = \bigoplus_{k=1}^{n}(l_{ik}\otimes u_{ki})\\ \nonumber
		&\ = \max_{k=1}^{n}(l_{ik}+u_{ki})\\
		&\ = \max_{k=1}^{i}(a_{ik}-a_{kk}+a_{ki})\\
		&\ =a_{ii}.
		\end{align}
		The equality $ (4) $ holds since for all $k> i$, $u_{ki}=l_{ik}=-\infty$. Furthermore, with respect to any permutation $\sigma =(ik)$, we have
		\begin{center}
			$a_{1\sigma(1)}\otimes \cdots \otimes a_{n\sigma (n)} \leq det_{\varepsilon}(A) $ \qquad i.e.,
		\end{center}
		\begin{center}
			$\displaystyle{ a_{ik}+a_{ki}+ \sum _{
					\begin{subarray}{c}
					j=1,\\j\neq k, i
					\end{subarray}}^{n} a_{jj} \leq a_{ii}+  a_{kk}+ \sum _{\begin{subarray}{c}
					j=1, \\ j\neq k, i
					\end{subarray}}^{n} a_{jj} }.$
		\end{center}
		Since $S$ is a multiplicatively cancellative semifield, we get $ a_{ik}+a_{ki}= a_{ii} + a_{kk} $ which yields the equality $ (5 )$.
		\item  For other entries $(i,j > 1; i\neq j)$, we have
		\begin{align}
		(LU)_{ij} &\ =\nonumber \bigoplus_{k=1}^{n}(l_{ik} \otimes u_{kj})\\ \nonumber
		&\ = \max_{k=1}^{n}(l_{ik}+u_{kj})\\
		&\ = \max_{k=1}^{min\{i,j\}}(l_{ik}+u_{kj}) \\ \nonumber
		&\ = \max_{k=1}^{min\{i,j\}}(a_{ik}- a_{kk}+ a_{kj})\\
		&\ = \max (\max_{k=1}^{min\{i,j\}-1}(a_{ik}- a_{kk}+ a_{kj}), a_{ij})\\ \nonumber
		&\ =  \max_{k=1}^{min\{i,j\}-1}(\max(a_{ik}+ a_{kj}-a_{kk}), (a_{ij}+a_{kk}-a_{kk}))\\ \nonumber
		&\ = \max_{k=1}^{min\{i,j\}-1}(\max(a_{ik}+ a_{kj}, a_{ij}+a_{kk})-a_{kk})\\ \nonumber
		&\ = \max_{k=1}^{min\{i,j\}-1}(det_{\varepsilon}(A[\{k,i\} \mid \{k,j\}])-a_{kk})\\ \nonumber
		&\ = a_{ij},
		\end{align}
		The equality $(6)$ holds, because
		\begin{center}
			$(\forall k)(k > i, k > j \Rightarrow l_{ik}= -\infty , u_{kj}= -\infty)$. \end{center}
		As a result, for any $k > \min \{i,j\} $, $l_{ik}+u_{kj}=-\infty$. The equality $(7)$ holds, because for $k= \min\{i,j\}$ ,  $a_{ik}-a_{kk}+a_{kj}=a_{ij}$.\\
		
		Conversely, assume that $ A $ has the presented $ LU$-factorization, that is $a_{ij} =(LU)_{ij} $. Then by~\eqref{L} and~\eqref{U}, the proof is complete.
	\end{enumerate}
\end{proof}
\begin{ex} Consider $ A\in M_{4}(S)$ as follows.
	\[
	A=\left[
	\begin{array}{cccc}
	7&-1&3&0\\
	4&5&1&-2\\
	1&-6&2&-5\\
	-2&-9&-5&0\\
	\end{array}
	\right]	.
	\]
	$A$ has the following $LU$ factors:\\
	\[
	L=\left[
	\begin{array}{cccc}
	0&-\infty&-\infty&-\infty\\
	-3&0&-\infty&-\infty\\
	-6&-11&0&-\infty\\
	-9&-14&-7&0\\
	\end{array}
	\right]	,~ U=
	\left[
	\begin{array}{cccc}
	7&-1&3&0\\
	-\infty&5&1&-2\\
	-\infty&-\infty&2&-5\\
	-\infty&-\infty &-\infty&0\\
	\end{array}
	\right].
	\]
\end{ex}

Note that a Maple procedure for calculating the $LU$ factors of a square matrix from Theorem~\ref{LUable} is given in Table~4.
%%%%%%%%%%%%%%%%%%%%%%%%%%%%%%%%%%%%%%%%5
\section{Solving a linear system of equations}
In this section, we first discuss and analyze a system of the form $LX= b$ with a lower triangular matrix, $L$, followed by a system of the form $UX= b$ with an upper triangular matrix, $U$. The combination of these results gives the solutions of the system $AX= b$ if the $ LU$ factors of $ A $ exist.
\subsection{L-system}
Here, we study the solution of the lower triangular system $LX=b$ where $L \in M_{n}(S)$ and $b \in S^{n}$ is a regular vector. The $i-$th equation of this system is
\begin{center}
	$\max(l_{i1}+x_{1}, l_{i2}+x_{2}+...+l_{ii}+x_{i},-\infty)= b_{i}$.
\end{center}
\begin{theorem}\label{Lthm}
	Let $LX=b $ be a linear system of equations with a lower triangular matrix $ L \in M_{n} (S)$ and a regular vector $ b \in S^{n}$. Then the system $LX=b $ has solutions $ X\leq(b_{i}-l_{ii} )_{i=1} ^{n}$ if $ l_{ik}-l_{kk}\leq b_{i}-b_{k} $ for any $ 2\leq i\leq n$ and $1\leq k\leq i-1$.\\
	Moreover, if all the inequalities $ l_{ik}-l_{kk}\leq b_{i}-b_{k} $ are proper, then the maximal solution $ X= (b_{i}-l_{ii} )_{i=1} ^{n}$  of the system $ LX=b $ is unique.
\end{theorem}
\begin{proof}	
	The proof is through induction on $ i $. For $ i=2~(k=1)$, the second equation of the system $ LX=b $ in the form
	$$ max( l_{21}+ x_{1},l_{22} + x_{2} ) =b_{2} $$
	implies that $ x_{2}\leq b_{2} - l_{22} $, since $ l_{21}-l_{11}\leq b_{2}-b_{1} $ and $ x_{1}=b_{1}- l_{11} $. We also show that the statement is true for $ i=3\quad (k= 1,2) $. Since the inequalities $ l_{31}- l_{11}\leq b_{3} -b_{1} $ and $ l_{32}- l_{22}\leq b_{3} -b_{2} $ hold, replacing for $ x_{1}$, and $x_{2} $ in the third equation yields $x_{3} \leq b_{3} - l_{33}$.
	
	Suppose that the statements are true for all $ i \leq m-1 $ , i.e., $ x_{1}=b_{1}-l_{11} $ and $ x _{i} \leq b_{i}-l_{ii} $, for any $ 2 \leq i \leq m-1 $.
	
	Now, let $ i=m~ (k=1, \cdots ,m-1) $. Then $ l_{ik}- l_{kk} \leq b_{i} -b_{k} $, and by the induction hypothesis, $ x _{i} \leq b_{i}-l_{ii} $ for any $2\leq i\leq m-1 $. As such, in the $ m- $th equation of the system we get $ x_{m} \leq b_{m} -l_{mm} $. Hence, the system $ LX=b $ has the maximal solution  $ X=(b_{i}-l_{ii} )_{i=1}^{n} $. Clearly, if $ l_{ik}-l_{kk} < b_{i}-b_{k} $ for any $ 2\leq i\leq n$ and $1\leq k\leq i-1$, then the maximal solution $X= (b_{i}-l_{ii} )_{i=1} ^{n}$ is unique.
\end{proof}
\begin{remark}
	The inequalities in Theorem~\ref{Lthm} give a sufficient condition  for the existence of the solutions of $ LX=b $, but it is not a necessary condition.\\
	For example, let $ L \in M_{3}(S) $ be a lower triangular matrix and $ LX=b $ has solutions $ X\leq(b_{i}-l_{ii} )_{i=1}^{3} $. Then the equation $ \max( l_{31}+x_{1},l_{32}+x_{2},l_{33}+x_{3})=b_{3} $ implies $ l_{32}+x_{2}\leq b_{3} $. However, considering $x_{2}\leq b_{2}- l_{22} $, we can not necessarily conclude that $ l_{32}-l_{22}\leq b_{3}- b_{2} $.
\end{remark}
The next example shows that if $ l_{ik}-l_{kk} >b_{i}-b_{k} $ for some  $ 2\leq i\leq n $ and $ 1\leq k\leq i-1$, then the system $ LX=b $ has no solution.\\
\begin{ex}
	Consider the following system:
	\[
	\left[
	\begin{array}{ccc}
	3&-\infty&-\infty\\
	-5&4&-\infty\\
	8&18&-2
	\end{array}
	\right]
	\left[
	\begin{array}{c}
	x_{1}\\
	x_{2}\\
	x_{3}
	\end{array}
	\right]
	=
	\left[
	\begin{array}{c}
	6\\
	-2\\
	10
	\end{array}
	\right].
	\]
	It has no solution because $ l_{31} -l_{11} >b_{3}-b_{1} $ and $ x_{1}=b_{1}-l_{11} $. Consequently, $ l_{31} +x_{1} >b_{3}$.
\end{ex}
In the following theorem, we present necessary and sufficient conditions such that the lower triangular system $ LX=b $ has no solution.
\begin{theorem}\label{Lnothem}
	Let $LX=b $ be a linear system of equations with a lower triangular matrix $ L \in M_{n} (S)$ and a regular vector $ b \in S^{n}$. Then the system $LX=b $ has no solution if and only if $l_{ik}-l_{kk} > b_{i}- b_{k}$
	and $b_{i} - l_{ik} < x_{k} \leq b_{k} - l_{kk}$ for some $2\leq i\leq n$, $1\leq k\leq i-1$.
\end{theorem}
\begin{proof}
	Suppose that $l_{ik}-l_{kk} > b_{i}- b_{k}$ and $b_{i} - l_{ik} < x_{k} \leq b_{k} - l_{kk}$ for some $2\leq i\leq n$, $1\leq k\leq i-1$. Then in the $i-$th equation of the system $ LX=b $, we have  $ l_{ik}+x_{k} > b_{i}$, which yields the system has no solution.
	
	Conversely, assume that $LX=b$ has no solution. Then due to Theorem~\ref{Lthm}, there are some $2\leq i\leq n$, $1\leq k\leq i-1$ such that  $l_{ik} - l_{kk} > b_{i}- b_{k}$ . Let $i $ be the lowest index satisfying the last inequality. Then we have $x_{k} \leq b_{k} - l_{kk}$, for any $1\leq k\leq i-1$. Suppose that $x_{k} \leq b_{i} - l_{ik}$ for any $2\leq i\leq n$, $1\leq k\leq i-1$. This means the system $LX=b$ must have a solution, which is a contradiction. Hence, $b_{i} - l_{ik} < x_{k} \leq b_{k} - l_{kk}$ for some $2\leq i\leq n$, $1\leq k\leq i-1$.
\end{proof}
\begin{corollary}
	The system $LX=b$ has no solution if and only if $l_{i1} - l_{11} > b_{i} - b_{1}$, for some $ 2 \leq i  \leq n $.
\end{corollary}
\begin{proof}
	Suppose that $l_{i1} - l_{11} > b_{i} - b_{1}$, for some $ 2 \leq i  \leq n $. Then due to Theorem~\ref{Lnothem}, the system $ LX=b $ has no solution since $x_{1}= b_{1} - l_{11}$ and consequently, $ b_{i}- l_{i1} < x_{1} \leq b_{1} - l_{11} $, for some $  2 \leq i  \leq n $ and $ k=1 $. The converse is trivial.
\end{proof}
See Table~5 for a Maple procedure associated with the results of Theorems~\ref{Lthm} and~\ref{Lnothem}.
%%%%%%%%%%%%%%%%%%%%%%%%%%%%%%%%%%%%%%%%%%%%%%%%%%%%%%%%%%%%%%%%%%%%%%%%%%%%5
\subsection{U-system}
In this section, we study the solution of the upper triangular system $ UX=b $ with $ U \in M_{n}(S) $ and $ b \in S^{n} $ that is a regular vector. Note that we can rotate an upper triangular matrix and turn it into a lower triangular matrix through a clockwise $ 180-$degree rotation. As such, the $U-$system $ UX=b $
becomes an $L-$system $ LX^{\prime}=b^{\prime}  $ with the following matrix:
\[
\left[
\begin{array}{cccc}
u_{nn}&-\infty &\cdots&-\infty\\
u_{(n-1)n}&u_{(n-1)(n-1)}&\cdots&-\infty\\
\vdots&\vdots&\ddots&\vdots\\
u_{1n}&u_{1(n-1)}&\cdots&u_{11}
\end{array}
\right]
\] \\
where $ l_{ij}= u_{(n-i+1)(n-j+1)} $ , $ x_{i}^{\prime} = x_{n-i+1} $ and $ b_{i}^{\prime} =b_{n-i+1} $, for every $1 \leq i,j \leq n$ and $j \leq i$.
\begin{theorem}\label{Uthm}
	Let $ UX=b $ be a linear system of equations with an upper triangular matrix $U \in M_{n}(S)$ and a regular vector $b \in S^{n}$. Then the system $UX=b $ has solutions $X \leq (b_{j}- u_{jj})_{j=1}^{n}$ if $u_{(n-i)k}-u_{kk} \leq b_{(n-i)}- b_{k}$ for any $1\leq i\leq n-1$ and $n-i+1\leq k\leq n$. Moreover, if the inequalities $u_{(n-i)k}-u_{kk} \leq b_{(n-i)}- b_{k}$ are proper, then the maximal solution $X=(b_{j}- u_{jj})_{j=1}^{n}$ of the $U$-System is unique.
\end{theorem}
\begin{proof}
	We first convert the upper triangular system $UX=b$ into a lower triangular system $LX^{\prime}=b^{\prime}$ as explained above. We can now rewrite $u_{(n-i)k}-u_{kk} \leq b_{(n-i)}- b_{k}$ as
	$$l_{(i+1)(n+1-k)}- l_{(n+1-k)(n+1-k)} \leq b_{i+1}- b_{n+1-k}.$$
	Consequently, without loss of generality, we have $l_{st}-l_{tt} \leq b_{s}- b_{t}$ for every $2\leq s\leq n$, $1 \leq t\leq i-1$. Hence by Theorem~\ref{Lthm}, the system $LX^{\prime}=b^{\prime}$ has solutions $X^{\prime}=(x_{i}^{\prime})_{i=1}^{n}$ where $x_{i}^{\prime} \leq b_{i}^{\prime} - l_{ii}$ which implies $x_{n-i+1} \leq b_{n-i+1}-u_{(n-i+1)(n-i+1)}$. Thus, $UX=b$ has the maximal solution $X=(b_{j}- u_{jj})_{j=1}^{n}$. Note further that the maximal solution is unique if $u_{(n-i)k}-u_{kk} < b_{(n-i)}- b_{k}$ for any $1\leq i\leq n-1$ and $n-i+1\leq k\leq n$.
\end{proof}
The next theorem provides necessary and sufficient conditions such that the $U$-system $ UX=b $ has no solution.
\begin{theorem}\label{Unothm}
	Let $ UX=b $ be a linear system of equations with an upper triangular matrix $U \in M_{n}(S)$ and a regular vector $b \in S^{n}$. Then the system $UX=b$ has no solution if and only if $u_{(n-i)k}- u_{kk} > b_{n-i} -b_{k}$ and $b_{n-i} - u_{(n-i)k} \leq x_{k} \leq b_{k}-u_{kk}$ for some $1\leq i\leq n-1$, $n-i+1\leq k\leq n$.
\end{theorem}
\begin{proof}
	By converting the system $UX=b$ into the system $LX^{\prime}=b^{\prime}$ and applying Theorem~\ref{Lnothem}, the proof is complete.
\end{proof}
%%%%%%%%%%%%%%%%%%%%%%%%%%%
A Maple procedure based on the results of Theorems~\ref{Uthm} and~\ref{Unothm} is given in Table~6.
%%%%%%%%%%%%%%%%%%%%%%%%%%%%%%%%%%%%%%%%%%%%%%%%%%%%%%%%%%%%%%%%
\subsection{LU-system}
The $LU$-factorization method is well-known for solving systems of linear equations. In this section, we find the solutions of the system $ AX=b $ for any $ A \in M_{n}(S) $ and a regular vector $ b \in S^{n} $, where $ A $ has $LU$ factors.
\begin{theorem}\label{LUthm}
	Let $ A \in M_{n}(S) $ and suppose $ A $ has $LU$ factors given by~\eqref{L} and~\eqref{U}, respectively. The system $ AX=b $ has the maximal solution $ X= (b_{i}-a_{ii})_{i=1}^{n}$ if $ a_{ik}-a_{kk}\leq b_{i}-b_{k} $ and $ a_{(n-j)l} -a_{ll}\leq b_{(n-j)} -b_{l} $ for any  $2\leq i\leq n$, $1\leq k\leq i-1$, $1\leq j\leq n-1$, and $n-j+1\leq l\leq n$. Furthermore, the system $ AX=b $ has the unique solution  $ X= (b_{i}-a_{ii})_{i=1}^{n}$ if all the iequalities $ a_{ik}-a_{kk} \leq b_{i}-b_{k} $ and $ a_{(n-j)l} -a_{ll} \leq b_{(n-j)} -b_{l} $ are proper.
\end{theorem}
\begin{proof}
	Let the matrix $A$ have $LU$ factors. Then the system $ AX=b $ may be rewritten as $ L(UX)=b $. To obtain $ X $, we must first decompose $ A $ and then solve the system $ LZ=b $ for $ Z $, where $ UX= Z$. Once $ Z $ is found, we solve the system $ UX=Z $ for $ X $.\\
	By the definition of  $L$, we have $l_{ik}= a_{ik}-a_{kk}$, for any $ i,k \in \underline{n} $ and $ i> k$. As such, the inequalities $ a_{ik}-a_{kk}\leq b_{i}-b_{k}$ for any $2\leq i\leq n$ and $1\leq k\leq i-1$ can be rewritten as the inequalities $l_{ik}- l_{kk} \leq b_{i}- b_{k}$. Therefore, due to Theorem~\ref{Lthm}, the system $LZ=b$ has solutions $Z \leq (b_{i} - l_{ii})_{i=1}^{n}$ where $l_{ii}=0$. Similarly, we can turn the inequalities $ a_{(n-j)l} -a_{ll} \leq b_{(n-j)} -b_{l} $, for any $1\leq j\leq n-1$, and $n-j+1\leq l\leq n$ into the inequalities $ u_{(n-j)l} -u_{ll} \leq b_{(n-j)} -b_{l} $, since $u_{ij}= a_{ij}$, for any $ i,j \in \underline{n}$ and $ i<j $. Then due to Theorem~\ref{Uthm}, the system $UX=Z$ has solutions $X \leq ( z_{i}- a_{ii})_{i=1}^{n}$. Consequently, the system $AX=b $ has the maximal solution $X = (b_{i}- a_{ii})_{i=1}^{n}$. Clearly, the obtained maximal solution is unique whenever all the assumed inequalities are proper.
\end{proof}
\begin{ex}	
	Let $ A\in M_{4}(S) $. Consider the following system $ AX=b $:	
	\[
	\left[
	\begin{array}{cccc}
	4&1&4&3\\
	-1&0&1&4\\
	3&7&8&1\\
	5&2&5&-2
	\end{array}
	\right]
	\left[
	\begin{array}{c}
	x_{1}\\
	x_{2}\\
	x_{3}\\
	x_{4}
	\end{array}
	\right]
	=
	\left[
	\begin{array}{c}
	3\\
	4\\
	9\\
	4
	\end{array}
	\right].
	\]
	Here, $det_{\varepsilon}(A)=a_{13}+a_{24}+a_{32}+a_{41}$, but there exists a permutation matrix $ P_{\sigma} $ corresponding to the permutation $ \sigma= (1324) $:\\
	\[ P_{\sigma}=\left[
	\begin{array}{cccc}
	-\infty&-\infty&-\infty&0\\
	-\infty&-\infty&0&-\infty\\
	0&-\infty&-\infty&-\infty\\
	-\infty&0&-\infty&-\infty\\
	\end{array}
	\right],
	\]
	such that $ P_{\sigma}A $ has the following $ LU $ factors:
	\[
	L=\left[
	\begin{array}{cccc}
	0&-\infty&-\infty&-\infty\\
	-2&0&-\infty&-\infty\\
	-1&-6&0&-\infty\\
	-6&-7&-3&0\\
	\end{array}
	\right]	,~ U=
	\left[
	\begin{array}{cccc}
	5&2&5&-2\\
	-\infty&7&8&1\\
	-\infty&-\infty&4&3\\
	-\infty&-\infty &-\infty&4\\
	\end{array}
	\right];
	\]
	We can now use the $ LU $ method to solve the system $ (P_{\sigma}A)X = P_{\sigma}b$:
	\[
	\left[
	\begin{array}{cccc}
	5&2&5&-2\\
	3&7&8&1\\
	4&1&4&3\\
	-1&0&1&4
	\end{array}
	\right]
	\left[
	\begin{array}{c}
	x_{1}\\
	x_{2}\\
	x_{3}\\
	x_{4}
	\end{array}
	\right]
	=
	\left[
	\begin{array}{c}
	4\\
	9\\
	3\\
	4
	\end{array}
	\right]
	\].
	Due to Theorem~\ref{LUthm}, we must first solve the system $ LZ=P_{\sigma}b $:
	\[
	\left[
	\begin{array}{cccc}
	0&-\infty&-\infty&-\infty\\
	-2&0&-\infty&-\infty\\
	-1&-6&0&-\infty\\
	-6&-7&-3&0\\
	\end{array}
	\right]
	\left[
	\begin{array}{c}
	z_{1}\\
	z_{2}\\
	z_{3}\\
	z_{4}
	\end{array}
	\right]
	=
	\left[
	\begin{array}{c}
	4\\
	9\\
	3\\
	4
	\end{array}
	\right]
	\].
	Since $ l_{ik} - l_{kk}\leq b^{\prime}_{i} - b^{\prime}_{k} $ for any $ 2 \leq i \leq 4 $ and $ 1 \leq k \leq i - 1 $, the solutions are $ Z \leq (b^{\prime}_{i} - l_{ii})_{i=1}^{4} =(b^{\prime}_{i})_{i=1}^{4} $, where $ b^{\prime}= P_{\sigma}b $:
	\[
	Z \leq \left[
	\begin{array}{c}
	4\\
	9\\
	3\\
	4
	\end{array}
	\right]
	\]
	We shall now solve the system $ UX=Z $:
	\[
	\left[
	\begin{array}{cccc}
	5&2&5&-2\\
	-\infty&7&8&1\\
	-\infty&-\infty&4&3\\
	-\infty&-\infty &-\infty&4\\
	\end{array}
	\right]
	\left[
	\begin{array}{c}
	x_{1}\\
	x_{2}\\
	x_{3}\\
	x_{4}
	\end{array}
	\right]
	=
	\left[
	\begin{array}{c}
	z_{1}\\
	z_{2}\\
	z_{3}\\
	z_{4}
	\end{array}
	\right]
	\].
	By Theorem~\ref{Uthm}, the solutions are $ X \leq (z_{i} -u_{ii})_{i=1}^{4} $ or in fact $ X \leq (b^{\prime}_{i} - u_{ii})_{i=1}^{4} $, since $u_{(4-i)k}-u_{kk} \leq b_{(4-i)}- b_{k}$ for any $1\leq i\leq 3$ and $ 4-i+1\leq k\leq 4$ :
	\[
	X \leq \left[
	\begin{array}{c}
	-1\\
	2\\
	-1\\
	0
	\end{array}
	\right]
	\].
	Note that the systems $AX=b$ and $(P_{\sigma}A)X=P_{\sigma}b$ have the same solutions.
\end{ex}
\begin{remark}
	Theorem~\ref{LUthm} shows that if the system $(LU)X=b$ does not have any solutions, then either the system $LZ=b$ or the system $UX=Z$ must not have any solutions.
	%Now, what we can say about the solution of the system $AX=b$ when one of systems $LZ=b$ or $UX=Z$ or both of them has no solutions.
\end{remark}
The next example shows that if the system $LZ=b$ has no solution, then neither does the system $AX=b$.
\begin{ex}	
	Let $ A\in M_{4}(S) $. Consider the following system $ AX=b $:	
	\[
	\left[
	\begin{array}{cccc}
	7&-1&3&0\\
	4&5&1&-2\\
	1&-6&2&-5\\
	-2&-9&-5&0
	
	\end{array}
	\right]
	\left[
	\begin{array}{c}
	x_{1}\\
	x_{2}\\
	x_{3}\\
	x_{4}
	\end{array}
	\right]
	=
	\left[
	\begin{array}{c}
	5\\
	2\\
	-1\\
	-9
	\end{array}
	\right]
	\].
	By Theorem~\ref{LUable}, $ A $ has an $LU-$factorization. It is fairly straightforward to verify that $ LZ=b $ has no solution, because $ l_{41}> b_{4}-b_{1}$. We have $a_{41}+x_{1} > b_{4}$, since $x_{1}= b_{1}-a_{11}$ and $l_{41}=a_{41}-a_{11}$ which implies $ AX=b $ not have any solutions.
\end{ex}
\begin{theorem}\label{LUnothm}
	Let $A \in M_{n}( S)$. Suppose $A$ has an $LU$-factorization. If the system $LZ=b$ has no solution, then neither does the system $AX=b$.
\end{theorem}
\begin{proof}
	Assume that $LZ=b$ does not have any solutions. Then Theorem~\ref{Lnothem} implies that $l_{ik}- l_{kk} > b_{i}- b_{k}$ and $b_{i}- l_{ik} < z_{k} \leq b_{k}- l_{kk}$ for some $2\leq i\leq n$ and $1\leq k\leq i-1$. As such, we have $ l_{ik} + z_{k} > b_{i}$, where $z_{k}$ is obtained from the $k$-th equation of the system $UX=Z$. That means
	\begin{align}
	l_{ik}+z_{k}
	&\ \nonumber = l_{ik} + \max(u_{kk}+x_{k},\cdots,u_{kn}+x_{n}) \\
	&\ =\max (a_{ik}-a_{kk}+a_{kk}+x_{k},\cdots, a_{ik}-a_{kk}+a_{kn}+x_{n}).
	\end{align}	
	Since the matrix $A$ has an $LU$-factorization, then Theorem~\ref{LUable} leads to
	\begin{align*}
	(LU)_{ij}
	&\ = \max_{t=1}^{r}( det_{\varepsilon}(A[\{t,i\}] \mid \{t,j\}])- a_{tt})\\
	&\ = \max_{t=1}^{r}(\max (a_{ij}, a_{tj} -a_{tt}+a_{it}))\\
	&\ = a_{ij},
	\end{align*}
	where $r=\min \{ i,j \}-1$. As such, $a_{ik}-a_{kk}+a_{kj} \leq a_{ij}$, for any $k+1\leq j\leq n$, since $1\leq k\leq i-1$ $(k < \min\{i,j \})$. Due to~$(8)$, the following inequality is obtained
	\begin{center}
		$l_{ik}+z_{k} \leq \max (a_{ik}+x_{k},a_{i(k+1)}+x_{k+1},\cdots, a_{in}+x_{n})$.
	\end{center}
	Moreover, in the $i$-th equation of the system $LZ=b$, we have $l_{ik}+z_{k} > b_{i}$. Consequently, in the $i$-th equation of the system $AX=b$, we get $\max (a_{i1}+x_{1},\cdots, a_{in}+x_{n}) > b_{i}$ which means $AX=b$ has no solution.
\end{proof}
\begin{remark}
	Take the system $AX= b$. Let $A \in M_{n}(S)$ have an $LU$-factorization and $UX=Z$. If $LZ=b$ has some solution, but $UX=Z$ does not have any solutions, then $AX=b$ has no solution.
	%Generally, if $LZ=b$ or $UX=Z$ or both of them have no solution, then $UX=Z$ has no solution.?
\end{remark}

\section{Extension of the idea to idempotent semifields}	
Throughout this section, let $S$ be an idempotent semifield and $A=(a_{ij}) \in M_{n}(S)$ with $det_{\epsilon}(A)=a_{11}\otimes ...\otimes a_{nn}$.  Take a lower triangular matrix $L$ and an upper triangular matrix $U$ over $S$ as follows.
\begin{equation*}
L=(l_{ij});~~
l_{ij}= \left
\{
\begin{array}{cc}
a_{ij}\otimes a_{jj}^{-1} & i\geq j\\
0 & \: i <j
\end{array}
\right.
\end{equation*}
and
\begin{equation*}
U=(u_{ij});~~
u_{ij}= \left
\{
\begin{array}{cc}
a_{ij} & i\leq j\\
0 & i > j
\end{array}
\right.
\end{equation*}\\
We say $A$ has an $LU$-factorization if $A=LU$.
%%%%%%%%%%%%%%%%%%%%%%%%%%%%%%%%%%%%%%%%%%%%%%%%%%%%%%%%%%%%%%%%
\begin{theorem}
	Let $A \in M_{n}(S)$ and $n$ be a positive integer such that $det_{\varepsilon}(A)=a_{11} \otimes....\otimes a_{nn}$ and matrices $L$ and $U$ be defined as above. Then $A=LU$ if and only if for any $ 1 < i, j \leq n $ and $i\neq j$,
	$$a_{ij}= \bigoplus_{k=1}^{r} ( det_{\varepsilon} (A[\{k,i\} \mid \{k,j\}])\otimes a_{kk}^{-1}),$$
	where $r=min\{i,j\}-1$.
\end{theorem}
\begin{proof}
	Note that  $ S $ is additively idempotent and the total order $ ``\leq_{S}" $ on $ S $ is  compatible with addition and multiplication. Moreover, the total order $ ``\leq_{S}" $ induces a partial order on $ S $, which is stated in Section~2. Considering these properties, the proof is similar to that of Theorem~\ref{LUable}.
\end{proof}
%%%%%%%%%%%%%%%%%%%%%%%%%%%%%%%%%%%%%%%%%%%%%%%%%%%%
\begin{theorem}\label{ELthm}
	Let $LX=b $ be a linear system of equations with a lower triangular matrix $ L \in M_{n} (S)$ and a regular vector $ b \in S^{n}$. Then the system $LX=b $ has solutions $ X \leq _{S} (b_{i}\otimes l_{ii}^{-1} )_{i=1} ^{n}$ if $ l_{ik} \otimes l_{kk}^{-1} \leq _{S} b_{i} \otimes b_{k}^{-1} $, for any $ 2\leq i\leq n$ and $ 1\leq k\leq i-1$.\\
	Moreover, if all the inequalities are proper, then the maximal solution $ X=(b_{i} \otimes l_{ii}^{-1}) $ of the system is unique.
\end{theorem}
\begin{proof}
	The proof is by induction on $ i $. For $ i=2 ~ (k=1), $ since $ l_{21}\otimes l_{11}^{-1} \leq _{S} b_{2} \otimes b_{1}^{-1} $ and $x_{1}=b_{1}\otimes l_{11}^{-1} $, we have $(l_{21}\otimes x_{1}) \leq _{S} b_{2}$. As such, the second equation of the system, $ (l_{21}\otimes x_{1}) \oplus (l_{22} \otimes x_{2}) =b_{2}, $ implies  $ x_{2} \leq _{S} b_{2} \otimes l_{22}^{-1}  $.\\
	Suppose that for $ i \leq m-1 $, the statements are true , i.e. $ x_{1}=b_{1}\otimes l_{11}^{-1} $ and $ x _{i} \leq _{S} b_{i}\otimes l_{ii}^{-1} $, for any $2 \leq i \leq m-1 $ . Let  $ i=m,(k=1, \cdots ,m-1)$. Then the proof is similar to that of Theorem~\ref{Lthm}. Hence, the system $ LX=b $ has solutions  $ X \leq _{S} (b_{i} \otimes l_{ii}^{-1})$.
\end{proof}
%%%%%%%%%%%%%%%%%%%%%%%%%%%%%%%%%%%%
\begin{remark}\label{remextend}
	Similarly, we can extend Theorems~\ref{Lnothem}, \ref{Uthm}, \ref{Unothm}, \ref{LUthm} and \ref{LUnothm} to idempotent semifields as follows.
\end{remark}
\begin{enumerate}
	\item Let $LX=b $ be a linear system of equations with a lower triangular matrix $ L \in M_{n} (S)$ and a regular vector $ b \in S^{n}$. Then the system $LX=b $ has no solution if and only if $l_{ik}\otimes l_{kk}^{-1} >_{S} b_{i}\otimes b_{k}^{-1}$
	and $b_{i} \otimes l_{ik}^{-1} <_{S} x_{k} \leq _{S} b_{k} \otimes l_{kk}^{-1}$ for some $2\leq i\leq n$ and $1\leq k\leq i-1$.
	%%%%%%%%%%%%%%%%%%%%%%%%%%%%%%%%%%%%%%%%%%%
	\item Let $ UX=b $ be a linear system of equations with an upper triangular matrix $U \in M_{n}(S)$ and a regular vector $b \in S^{n}$. Then the system $UX=b $ has solutions $X \leq_{S} (b_{j} \otimes u_{jj}^{-1})_{j=1}^{n}$ if $u_{(n-i)k} \otimes u_{kk}^{-1} \leq _{S} b_{(n-i)} \otimes b_{k}^{-1}$, for any $1\leq i\leq n-1$, $n-i+1\leq k\leq n$. Moreover, if the inequalities $u_{(n-i)k} \otimes u_{kk}^{-1} \leq _{S} b_{(n-i)} \otimes b_{k}^{-1}$ are proper, then the maximal solution $X=(x_{j} \otimes u_{jj}^{-1})_{j=1}^{n}$ of the U-System $UX=b $ is unique.
	%%%%%%%%%%%%%%%%%%%%%%%%%%%%%%%%%%%%%%%%%%%%%
	\item Let $ UX=b $ be a linear system of equations with an upper triangular matrix $U \in M_{n}(S)$ and a regular vector $b \in S^{n}$. Then the system $UX=b$ has no solution if and only if $u_{(n-i)k} \otimes u_{kk}^{-1} >_{S} b_{n-i} \otimes b_{k}^{-1}$ and $b_{n-i} \otimes u_{(n-i)k}^{-1} \leq _{S} x_{k} \leq _{S} b_{k} \otimes u_{kk}^{-1}$ for some $1\leq i\leq n-1$ and $n-i+1\leq k\leq n$.
	%%%%%%%%%%%%%%%%%%%%%%%%%%%%%%%%%%%%%%%%%%%%%%%%%%%
	\item	Let $ A \in M_{n}(S) $ and suppose $ A $ has $LU$ factors given by~\eqref{L} and~\eqref{U}, respectively. The system $ AX=b $ has the solutions $ X \leq _{S} (b_{i} \otimes a_{ii}^{-1})_{i=1}^{n}$ if $ a_{ik} \otimes a_{kk}^{-1} \leq _{S} b_{i} \otimes b_{k}^{-1} $ and $ a_{(n-j)l} \otimes a_{ll}^{-1} \leq _{S} b_{(n-j)} \otimes b_{l}^{-1} $, for every  $2\leq i\leq n$, $1\leq k\leq i-1$, $1\leq j\leq n-1$, and $n-j+1\leq l\leq n$. Furthermore, the system $ AX=b $ has the unique solution  $ X= (b_{i} \otimes a_{ii}^{-1})_{i=1}^{n}$ if the inequalities $ a_{ik} \otimes a_{kk}^{-1} \leq _{S} b_{i} \otimes b_{k}^{-1} $ and $ a_{(n-j)l} \otimes a_{ll}^{-1} \leq _{S} b_{(n-j)} \otimes b_{l}^{-1} $ are proper.
	%%%%%%%%%%%%%%%%%%%%%%%%%%%%%%%%%%%%%%%%%%%%%%%%%%%%
	\item	Let $A \in M_{n}( S)$. Suppose $A$ has an $LU$-factorization. If the system $LZ=b$ has no solution, then neither does the system $AX=b$.
	%%%%%%%%%%%%%%%%%%%%%%%%%%%%%%%%%%%%%%%%}
\end{enumerate}
\begin{ex}	
	Let $ A\in M_{4}(S) $ where $S=\mathbb{R}_{\min, \times}$. Consider the following system $ AX=b $:	
	\[
	\left[
	\begin{array}{cccc}
	1&6&9&8\\
	6&2&7&5\\
	9&7&1&7\\
	8&5&6&3
	\end{array}
	\right]
	\left[
	\begin{array}{c}
	x_{1}\\
	x_{2}\\
	x_{3}\\
	x_{4}
	\end{array}
	\right]
	=
	\left[
	\begin{array}{c}
	4\\
	6\\
	1\\
	6
	\end{array}
	\right],
	\]
	where $det_{\epsilon}(A)=a_{11} \times a_{22} \times a_{33} \times a_{44}=6$. We use the presented $ LU $-method to solve this system. Due to the extension of the Theorem~\ref{LUthm} in Remark~\ref{remextend}, we must first solve the system $ LZ=b $:
	\[
	\left[
	\begin{array}{cccc}
	1&+\infty&+\infty&+\infty\\
	6&1&+\infty&+\infty\\
	9&\frac{7}{2}&1&+\infty\\
	8&\frac{5}{2}&6&1\\
	\end{array}
	\right]
	\left[
	\begin{array}{c}
	z_{1}\\
	z_{2}\\
	z_{3}\\
	z_{4}
	\end{array}
	\right]
	=
	\left[
	\begin{array}{c}
	4\\
	6\\
	1\\
	6
	\end{array}
	\right]
	\]
	Since $ l_{ik} \otimes l_{kk}^{-1} \leq_{S} b_{i} \otimes b_{k}^{-1} $ for any $ 2 \leq i \leq 4 $ and $ 1 \leq k \leq i - 1 $, the solutions are $ Z\leq_{S} (b_{i} \otimes l_{ii}^{-1})_{i=1}^{4} =(b_{i})_{i=1}^{4} $:
	\[
	Z\leq_{S} \left[
	\begin{array}{c}
	4\\
	6\\
	1\\
	6
	\end{array}
	\right]
	\]
	We shall now solve the system $ UX=Z $:
	\[
	\left[
	\begin{array}{cccc}
	1&6&9&8\\
	+\infty&2&7&5\\
	+\infty&+\infty&1&7\\
	+\infty&+\infty &+\infty&3\\
	\end{array}
	\right]
	\left[
	\begin{array}{c}
	x_{1}\\
	x_{2}\\
	x_{3}\\
	x_{4}
	\end{array}
	\right]
	=
	\left[
	\begin{array}{c}
	z_{1}\\
	z_{2}\\
	z_{3}\\
	z_{4}
	\end{array}
	\right].
	\]
	By the extension of the Theorem~\ref{Uthm} in Remark~\ref{remextend}, the solutions are $ X \leq_{S} (z_{i} \otimes u_{ii}^{-1})_{i=1}^{4} $ or in fact $ X \leq_{S} (b_{i} \otimes u_{ii}^{-1})_{i=1}^{4} $, since $u_{(4-i)k} \otimes u_{kk}^{-1}  \leq_{S} b_{4-i} \otimes b_{k}^{-1}$ for any $1\leq i\leq 3$ and $ 4-i+1\leq k\leq 4$ :
	\[
	X \leq_{S} \left[
	\begin{array}{c}
	4\\
	3\\
	1\\
	2
	\end{array}
	\right].
	\]
	It should be noted  that $ a \leq_{S} b $ means that $ a\geq b $ for any $ a,b \in \mathbb{R}_{\min , \times} $, where $ ``\geq" $ is the standard greater than or equal relation over $ \mathbb{R_{+}}$.
\end{ex}
%%%%%%%%%%%%%%%%%%%%%%%%%%55
\section{Concluding Remarks}\label{remarks}
In this paper, we extended the $LU$-factorization technique to idempotent semifields. We stated the criteria under which a matrix can have $LU$ factors, and when it does, what the factors look like. Importantly, we used the results in solving linear systems of equations when the solution exists. One can use these $LU$ factors in the design of numerical algorithms in idempotent semifields. Other important properties of these $LU$ factors can also be studied especially relative to well-known classic results from linear algebra.

\newpage
\appendix
\begin{center}
	\begin{table}[!tb]
		\begin{verbatim}
			MaxPlusDet := proc (A::Matrix) 
		local i, j, s, n, detA, ind, K, V;
		description "This program finds the determinant of a square matrix in max-plus.";
		Use LinearAlgebra in
		n := ColumnDimension(A);
		V := Matrix(n);
		ind := Vector(n);
		if n = 1 then 
		   V := A[1, 1]; 
		   detA := V; 
		   ind[1] := 1 
		elif n = 2 then
		   V[1, 1] := A[1, 1]+ A[2, 2];
		   V[1, 2] := A[1, 2]+ A[2, 1]; 
		   V[2, 1] := A[1, 2]+ A[2, 1]; 
		   V[2, 2] := A[1, 1]+ A[2, 2]; 
		   detA := max(V);
		   for s to 2 do 
		     K := V[s, 1 .. 2]; 
		     ind[s] := max[index](K) 
		   end do;
		else 
		   for i to n do
		     for j to n do
		       V[i, j] := A[i, j]+ op(1, MaxPlusDet(A[[1 .. i-1, i+1 .. n], [1 .. j-1, j+1 .. n]])); 
		     end do; 
		     detA := max(V); 
		     K := V[i, 1 .. n];
		     ind[i] := max[index](K);
		   end do; 
		end if;
		end use:
		[detA, ind, V]
		end proc:
		\end{verbatim}
		\caption{Finding the determinant of a square matrix in max-plus} \label{tab1}
	\end{table}
\end{center}

\begin{center}
	\begin{table}[!tb]
		\begin{verbatim}
		Pmat := proc (A::Matrix) 
		local i, n, l, V, d, L, j, L1, P;
		description "This program finds the permutation matrix based on the determinant of a square matrix in max-plus.";
		use LinearAlgebra in
		n := ColumnDimension(A); 
		d := op(1, MaxPlusDet(A));
		P := Matrix(1 .. n, 1 .. n, (-1)*Float(infinity)); 
		V := op(3, MaxPlusDet(A));
		L := [];
		L1 := []; 
		for i to n do
		   for j to n do
		    if V[i, j] = d then
		       if j in L then 
		         L1 := [op(L1), j]
		       else
		         L := [op(L), j];
		         break 
		       end if; 
		     end if;
		   end do; 
		   P[j, i] := 0 
		end do;
		end use:
		P
		end proc:
		\end{verbatim}
		\caption{Finding the permutation matrix of a square matrix in max-plus} \label{tab2}
	\end{table}
\end{center}

\begin{center}
	\begin{table}[!tb]
		\begin{verbatim}
		Matmul := proc (A::Matrix, B::Matrix)
		local i, j, m, n, p, q, C, L;
		description "This program finds the multiplication of two matrices in max-plus.";
		Use LinearAlgebra in
		m := RowDimension(A);
		n := ColumnDimension(A); 
		p := RowDimension(B);
		q := ColumnDimension(B); 
		C := Matrix(m, q);
		if n <> p then
		   print('impossible');
		   break 
		else
		   for i to m do
		     for j to q do
		       L := [seq(A[i, k]+B[k, j], k = 1 .. n)];
		       C[i, j] := max(L)
		     end do
		   end do
		end if;
		end use:
		C 
		end proc:
		\end{verbatim}
		\caption{Calculation of matrix multiplication in max-plus} \label{tab3}
	\end{table}
\end{center}

\begin{center}
	\begin{table}[!tb]
		\begin{verbatim}
		maxLU := proc (A::Matrix) 
		local i, j, m, n, k, h, V, L, U, P, B, s, r;
		Use LinearAlgebra in
		n := ColumnDimension(A);
		L := Matrix(1 .. n, 1 .. n, (-1)*Float(infinity));
		U := Matrix(1 .. n, 1 .. n, (-1)*Float(infinity));
		P := Pmat(A);
		B := Matmul(P, A);
		for h to n do
		   L[h, 1] := B[h, 1]-B[1, 1];
		   U[1, h] := B[1, h];
		   U[h, h] := B[h, h];
		   L[h, h] := 0;
		end do; 
		for i from 2 to n do 
		   for j from 2 to n do
		     if i <> j then
	 	     r := min(i, j)-1;
		       V := Vector(r); 
		       for k to r do
		         V[k] := max(B[i, k]+B[k, j], B[i, j]+B[k, k])-B[k, k];
		       end do;
		       s[i, j] := max(V); 
		       if s[i, j] = B[i, j] then 
		         if j < i then
		           L[i, j] := B[i, j]-B[j, j];
		         elif i < j then
		           U[i, j] := B[i, j];
		         end if;
		       else 
		         print('No solution');
		         break
		       end if;
		     end if;
		   end do;
		end do;
		end use:
		[P, L, U] 
		end proc:
		\end{verbatim}
		\caption{Calculating the LU factors of a square matrix in max-plus} \label{tab4}
	\end{table}
\end{center}

\begin{center}
	\begin{table}[!tb]
		\begin{verbatim}
		maxLsystem:=proc(L::Matrix, b::Vector)
		local i, k, n, x, c, V;
		description "This program solves a lower triagular system in max-plus.";
		use LinearAlgebra in  
		n:= ColumnDimension(L);
		x := Vector(n);
		c[1] := 1; \#(Equality Flag)
		x[1] := b[1]-L[1, 1];
		for i from 2 to n do
		   for k  to i-1 do    
		     V:= Vector(i-1);   
		     if x[k]<b[i]-L[i, k] then
		       V[k]:= 1;
		       x[i]:= b[i]-L[i,i];
		     elif x[k]= b[i]-L[i, k] then     
		       x[i]:= b[i]-L[i,i];
	 	     print('x'[i]<= b[i]- L[i, i]);   
		     elif x[k]> b[i]-L[i, k] then     
		       if c[k]= 1 then     
		         print('no solution');     
		         break   
		       elif c[k]=0 then 
		         x[k]:= b[i]-L[i, k];     
		         x[i]:= b[i]- L[i, i];     
		         V[k]:= 1;
		         print('x'[k]<= b[i]- L[i, k];    
		       end if;   
		     end if;    
		     if max(V)= 1 then    
		       c[i]:= 1;    
		     else    
		       c[i]:=0;
		     end if:
		   end do:  
		end do:
		end use:
		x
		end proc:     
		\end{verbatim}
		\caption{Solving the system $LX= b$ in max-plus} \label{tab5}
	\end{table}
\end{center}

\begin{center}
	\begin{table}[!tb]
		\begin{verbatim}
		maxUsystem:=proc(U::Matrix, b::Vector)
		local i, k, n, x, c, V;
		description "This program solves an upper triagular system in max-plus.";
		use LinearAlgebra in  
		n:= ColumnDimension(U);
		x := Vector(n);
		c[n] := 1; \#(Equality Flag)
		x[n] := b[n]-U[n, n];
		for i from 2 to n do
		   for k  to i-1 do
		     V:= Vector(i-1);
		     if x[n+1- k]<b[n+1- i]-U[n+1- i, n+1- k] then 
		       V[n+1- k]:= 1;     
		       x[n+1- i]:= b[n+1- i]-U[n+1-i, n+1- i];   
		     elif x[n+1- k]= b[n+1- i]-U[n+1- i, n+1- k] then     
		       x[n+1- i]:= b[n+1- i]-U[n+1- i,n+1- i];
		       print('x'[n+1- i]<= b[n+1- i]- U[n+1- i, n+1- i]);
		     elif x[n+1- k]> b[n+1- i]-U[n+1- i, n+1- k] then     
		       if c[n+1- k]= 1 then     
		         print('no solution');     
		         break    
		       elif c[n+1- k]=0 then     
		         x[n+1- k]:= b[n+1- i]-U[n+1- i, n+1- k];     
		         x[n+1- i]:= b[n+1- i]- U[n+1- i, n+1- i];     
		         V[n+1- k]:= 1;     
		         print('x'[n+1- k]<= b[n+1- i]- U[n+1- i, n+1- k];  
		       end if;   
		     end if;    
		     if max(V)= 1 then   
		       c[n+1-i]:= 1;    
		     else     
		       c[n+1-i]:=0;   
		     end if:
		   end do:
		end do:  
		end use:
		x
		end proc:     
		\end{verbatim}
		\caption{Solving the system $UX= b$ in max-plus} \label{tab6}
	\end{table}
\end{center}

\end{document}